\documentclass{amsart}
\input xypic
\xyoption{all}
\xyoption{rotate}
\newcommand{\End}{\operatorname{End}}
\newcommand{\bmin}{\displaystyle{\bigoplus^-}}
\newcommand{\rigZ}{{]Z_s[_{1-\epsilon}}}
\newcommand{\rigS}{{]S[_{1-\epsilon}}}

\newcommand{\dd}{\textup{d}}

\newcommand{\aloga}{A_{\textup{log},1}}
\newcommand{\Ker}{\operatorname{Ker}}

\newcommand{\res}{\operatorname{Res}}

\newcommand{\dR}{{\textup{dR}}}
\newcommand{\hdr}{H_{\dR}}
\newcommand{\et}{\textup{\'et}}
\newcommand{\het}{H_{\et}}
\newcommand{\dlog}{\operatorname{d\!\log}}

\newcommand{\pair}[1]{{\left\langle #1 \right\rangle}}
\newcommand{\gpair}[1]{\pair{#1}_{\gl}}
\newcommand{\hpair}[1]{\pair{#1}_{h}}
\newcommand{\apair}[1]{\pair{#1}_{\textup{aux}}}
\newcommand{\nekovar}{Nekov\'a\v r}
\newcommand{\OO}{\mathcal{O}}

\newcommand{\Div}{\operatorname{div}}

\newcommand{\gl}{{\textup{gl}}}

\def\htt_#1{H_#1^\otimes}

\newcommand{\col}{\textup{Col}}

\newcommand{\Hom}{\operatorname{Hom}}

\def\acoln_#1{\OO_{\col,#1}}
\def\Ocoln_#1{\Omega_{\col,#1}^1}

\newcommand{\XX}{\mathcal{X}}

\newcommand{\tr}{\operatorname{tr}}

\newcommand{\isom}{\cong}
\newcommand{\Q}{\mathbb{Q}}

\newcommand{\Qp}{\Q_p}

\newcommand{\Xbar}{\overline{X}}
\newcommand{\cy}{Y}
\newcommand{\cz}{Z}
\newcommand{\cZ}{\mathcal{Z}}
\newcommand{\Cbar}{\overline{C}}
\newcommand{\Fbar}{\overline{F}}

\newcommand{\DR}{\operatorname{DR}}
\newcommand{\Dst}{\operatorname{D}_{\textup{st}}}
\newcommand{\Dcris}{\operatorname{D}_{\textup{cris}}}
\newcommand{\hst}{H_{\textup{st}}}
\newcommand{\gipair}[1]{\pair{#1}_{\gl,U_v-\cz}}

\newtheorem{theorem}{Theorem}[section]

\newtheorem{proposition}[theorem]{Proposition}

\newtheorem{lemma}[theorem]{Lemma}
\newtheorem{corollary}[theorem]{Corollary}
\theoremstyle{definition}
\newtheorem{definition}[theorem]{Definition}
\newtheorem{remark}[theorem]{Remark}

\newtheorem{assumption}[theorem]{Assumption}
\numberwithin{equation}{section}
\begin{document}
\title{$p$-adic heights and Vologodsky integration}
\author{Amnon Besser}
\address{
Department of Mathematics\\
Ben-Gurion University of the Negev\\
P.O.B. 653\\
Be'er-Sheva 84105\\
Israel
}
\begin{abstract}
  we extend the result of~\cite{Bes99a} to the case of curves with semi-stable reduction. In this case, one can replace Coleman integration by Vologodsky integration to extend the Coleman-Gross definition of a $p$-adic height pairing. we show that this pairing still equals the one defined by Nekov\'a\v r.
\end{abstract}

\maketitle

\section{Introduction}
\label{sec:intro}

Let $K$ be a finite extension of $\Qp$ and let $X$ be a smooth variety over $K$. When $X$ has an integral model $\XX$ over the ring of integers $\OO_K$ of $K$, then past work of the author and others~\cite{Bes98a,Bes98b,Bes-deJ98,Bes-deJ02,Bes10,Bes99a} shows a close connection between $p$-adic arithmetic invariants of $X$ or $\XX$, such as syntomic regulators and $p$-adic heights, and the theory of Coleman integration~\cite{Col82,Col-de88,Bes99}.

The present work is a first in a series of papers aiming at removing the good reduction assumption above by replacing Coleman integration with Vologodsky integration~\cite{Vol01}, which is (to some extent) an extension of Coleman integration to the bad reduction case. In this first paper we will extend the equivalence, proved in~\cite{Bes99a}, between the Coleman-Gross height pairing on curves~\cite{Col-Gro89} and the Nekov\'a\v r height pairing~\cite{Nek93} to the semi-stable case, modifying the Coleman-Gross height pairing by substituting Vologodsky integration for Coleman integration (The fact that one can use Vologodsky integration to extend the scope of the Coleman-Gross height pairing was observed already in~\cite[Section 7]{Bes00}).

\newcommand{\Gal}{\operatorname{Gal}}
To state the main theorem more precisely, suppose now that $C$ is a smooth complete curve over a number field $F$ having the property that at every $p$-adic completion $K$ it has a model $\XX$ over $\OO_K$ with a strictly semi-stable reduction. This implies that the \'etale cohomology group $\het^1(\Cbar,\Qp)$, viewed as a representation of $\Gal(\Fbar/F)$ (here the bars denote extension of scalars to an algebraic closure),  is semi-stable, in the sense of Fontaine, at every prime above $p$. As a result, given certain auxiliary data recalled in Section~\ref{sec:height},  There is a well-defined Nekov\'a\v r height pairing~\cite{Nek93}
\newcommand{\hnek}{h_{\operatorname{Nek}}}
\newcommand{\hgc}{h_{\operatorname{GC}}}
\begin{equation*}
  \hnek: \Div_0(C)\times \Div_0(C) \to \Qp \;,
\end{equation*}
where $\Div_0(C)$ is the group of degree zero divisors on $C$. Note that this is a slight reformulation - The original Nekov\'a\v r pairing is between cohomology classes and to get the pairing on divisors one applies the \'etale Abel-Jacobi map. However, the description using mixed extensions, which we are going to use for the comparison, already uses the geometric source of the cohomology classes so it is appropriate to reformulate things in this way (see~\cite{Bes99a}).

When the models $\XX$ above are actually smooth over $\OO_K$, Coleman and Gross~\cite{Col-Gro89} give a construction of a height pairing with the same data. It is a sum of local height pairings $h_v$, one for each place $v$ of $F$. for $v$ dividing $p$ the construction has two main ingredients. The first is a certain projection operator, which is in fact the logarithm for the universal vectorial extension of the Jacobian of $C$, completed at $v$. The second ingredient is Coleman integration for certain forms with logarithmic singularities.

As the logarithm for the universal vectorial extension does not require an integral model of any sort, the first ingredient of the construction exists without any assumptions. Coleman integration may be replaced by Vologodsky integration~\cite{Vol01}. In fact, since iterated integrals are not used, these integrals were defined before Vologodsky by Zarhin~\cite{Zar96} and Colmez~\cite{Colm96}. We thus obtain a trivial extension of the Coleman-Gross height pairing to the case of a general smooth $C$ without any assumptions on the reductions above $p$:
\begin{equation*}
  \hgc:  \Div_0(C)\times \Div_0(C) \to \Qp \;,
\end{equation*}

 Our main result is as follows:

\begin{theorem}\label{mainthm}
  Suppose that $C/F$ is a smooth complete curve over a number field $F$ and with semi-stable reduction at each prime above $p$. Then with the same auxiliary data used to define them, the Nekov\'a\v r and Coleman-Gross height pairings are the same: For any two $y,z\in \Div_0(C)$ we have $\hnek(y,z) = \hgc(y,z)$.
\end{theorem}

While this theorem is global, there is nothing global that needs to be said about its proof, as the part related to the comparison of local heights for places not above $p$ is already dealt with in~\cite{Bes99a}. The remaining problem, that of comparing local height pairings for a curve $X$ over a $p$-adic field $K$, splits into two subproblems according to the two ingredients of the Coleman-Gross construction alluded to above. Of the two, it is actually the one involving Coleman integration that requires little modifications, as already the treatment of~\cite{Bes99a} used the Zarhin-Colmez integration, which does not need good reduction. 

This leaves the problem of comparing the projection combing from the logarithm of the universal vectorial extension with the corresponding construction in~\cite{Nek93}. It is therefore again entirely in the domain of Zarhin-Colmez integration and one does not need in principle Vologodsky's techniques. However, it is not obvious how to compare this description with the essentially cohomological treatment of Nekov\'a\v r. For this reason we use in our proof two ingredients.
\begin{enumerate}
\item The description of the projection in terms of local indices and Vologodsky integration~\cite{Bes00}
\item The description of Vologodsky integration on curves with semi-stable reduction in terms of Coleman integration given in~\cite{Bes-Zer13}. 
\end{enumerate}
The local indices give rise to a certain pairing on $1$-forms, and its analysis in terms of Coleman integrals will be key not only to the present work, but also for further work on regulators.

The paper is structured as follows: In Section~2 we recall the two height pairings and compare the Vologodsky integration with the construction of \nekovar. In Section~3 we recall the relevant facts about the local and double indices in the Coleman and Vologodsky settings. Finally, in Section~4 we prove the key formula for the global double index, from which the comparison between the two projections follows easily.

We would like to thank Daniel Disegni whose interest in this problem gave rise to the present work. I would also like to thank the department of Mathematics at the Georgia Institute of Technology, where this work was completed.

The author was supported by Israel Science Foundation grant No. 1517/13.

\section{The two height pairings}
\label{sec:height}

In this section we quickly review the details needed from the definitions of the Nekov\'a\v r and Coleman-Gross pairings, discuss the extension of the latter using Vologodsky integration and explain the comparison between the pairings to indicate what needs to be proved in later sections.
Let again $C$ be a smooth complete curve over a number field $F$.

Both heights are defined first on a pair of divisors with disjoint supports and both extend to arbitrary pairs as they factor via the Jacobian. It is obviously sufficient to compare them for pairs with disjoint supports. They both depend on the following choices:
\begin{itemize}
\item  A ``global log''- a continuous idele class character
 \begin{equation*}
   \ell: \mathbb{A}_F^\times/F^\times \to \Q_p\;.
 \end{equation*}
\item for each $v|p$ a choice of a subspace $W_v\in \hdr^1(C\otimes
  F_v/F_v)$ complementary to the space of holomorphic forms.
\end{itemize}

As noted in~\cite{Bes99a}, for the Coleman-Gross height to be defined we must have, for each $v|p$, a factorization
\begin{equation}
  \label{eq:ass7}
   \xymatrix{
     {\OO_{F_v}^\times}  \ar[rr]^{\ell_v} \ar[dr]^{\log_v} & &   \Q_p\\
         & F_v\ar[ur]^{t_v}
           }
\end{equation}

Let $y$ and $z$ be two divisors of degree $0$ on $C$ with disjoint supports $\cy$ and $\cz$ respectively. By definition, the Coleman-Gross height pairing $\hgc(y,z)$ is a sum of local terms over all finite places $v$ of $F$,
\begin{equation*}
  \hgc(y,z) = \sum_v h_v(y,z)\;.
\end{equation*}
The Nekov\'a\v r pairing has several descriptions. With the choice of a particular mixed extension (see~\cite{Bes99a}) it is also a sum of local terms
\begin{equation*}
  \hnek(y,z) = \sum_v h_v^\prime(y,z)\;.
\end{equation*}
In both cases, the local terms at $v$ depend only on the completion of $C,y,z$ at $v$.
Furthermore, for all $v$ not above $p$ we have~\cite[Proposition~3.2]{Bes99a} $h_v(y,z) = h_v^\prime(y,z)$.

A more careful comparison is required for places above $p$. For such a place $v$ let $K=F_v$, denote $C\otimes F_v$ by $X$ and consider $y$ and $z$ as divisors on $X$. For simplicity, we drop now the subscript $v$ from all notation ($W$, $t$, etc.), as the analysis from this point onward will now be completely local.

We begin by describing the Coleman-Gross local height pairing $h=h_v$. Let $H^0(X,\Omega_{X/K}^1(\log \cy))$ be the space of meromorphic one forms with logarithmic poles along $\cy$. Coleman and Gross define a map
\begin{equation}\label{Psimap}
 \Psi: H^0(X,\Omega_{X/K}^1(\log \cy)) \to \hdr^1(X/K)
 \end{equation} 
 with the following properties:
\begin{enumerate}
\item It is the identity on
  $H^0(X,\Omega_{X/K}^1)= F^0 \hdr^1(X/K)$,
\item It vanishes on $\dlog(g)$ for $g\in \OO(X-\cy)^\times $.\label{dlogprop}
\end{enumerate}
In fact, they define the map for varying $\cy$ but we will not need this.

The map $\Psi$ and the subspace $W$ give a unique choice of a differential form $\omega_y \in H^0(X,\Omega_{X/K}^1(\log \cy))$ satisfying the two properties
\begin{enumerate}
\item The residue divisor of $\omega_y$ is $y$.
\item $\Psi(\omega_y) \in W$.
\end{enumerate}

Suppose now that $X$ has a smooth $\OO_K$ model, so that Coleman integrals exist on $X$, giving a primitive function for $\omega_y$ defined up to a constant. Coleman integration requires a choice of a branch of the logarithm, for which we use the branch $\log=\log_v$ defined in~\eqref{eq:ass7}. The local height pairing at $v$ can now be defined via
\begin{equation}\label{hform}
  h(y,z) = \tr\left(\int_z \omega_y\right)
\end{equation}
with $\tr=\tr_v$ as in~\eqref{eq:ass7}.

We can now remove the good reduction assumption by simply using Vologodsky integration, which provides the primitive for $\omega_y$ without any good reduction assumptions.
\begin{definition}
  Suppose that $C$ is as above but with no assumptions about the reduction above $p$. The extended Coleman-Gross height pairing $\hgc$ is defined exactly at the usual Coleman-Gross height pairing, but using Vologodsky integration instead of Coleman integration in~\eqref{hform}.
\end{definition}
\begin{proposition}
  The extended Coleman-Gross height pairing factors via the Jacobian.
\end{proposition}
\begin{proof}
The proof is just as for the Coleman-Gross pairing. Let $g$ be a rational function on $X$. By property \eqref{dlogprop} of $\Psi$ we have $\Psi(\dlog(g)) = 0$, so that $\omega_{(g)} = \dlog(g)$. Just as for Coleman integration, $\int \dlog(g) = \log(g)$ so that
\begin{equation*}
  h((g),E) = \tr \log(g(E)) = \ell_v(g(E))\;.
\end{equation*}
This also holds away from $p$~\cite[Prop.~1.2]{Col-Gro89} and we thus have
$\hgc((g),E) = 0$ by the fact that $\ell$ is an idele class character.
\end{proof}

Let us now turn to the description of the local Nekov\'a\v r pairing. It is defined in terms of the ``geometric mixed extension'' as follows. Let
   \begin{align*}
    V&=\het^1(\Xbar,\Qp(1)),\quad E_1=\het^1(\Xbar - \overline{\cy},\Qp(1)),\quad
    E_2=\het^1(\Xbar; \overline{\cz},\Qp(1))\\
    \intertext{(\'etale\ cohomology of
      $\Xbar$ relative to $\overline{\cz}$), recalling that the bar denotes extension of scalars to an algebraic closure, and}\\
    E&=\het^1(\Xbar- \overline{\cy}; \overline{\cz},\Qp(1))\;,
\end{align*}
so that there is the following diagram with exact rows and columns,
\cite[p. 159]{Nek93}:
\begin{equation}\label{mixext}
\xymatrix{
 & & 0 \ar[d] & 0 \ar[d] & \\
0 \ar[r]& B(1)\ar[r] \ar@2{-}[d] & E_2 \ar[r]^{\pi} \ar[d] & V\ar[r] \ar[d] & 0\\
0 \ar[r]& B(1)\ar[r]  & E \ar[r] \ar[d] & E_1 \ar[r] \ar[d] & 0\\
 &  & A \ar@2{-}[r] \ar[d] & A \ar[d] & \\
 &  & 0  & 0  & 
}
\end{equation}

with
$A=(\Q_p^{\overline{\cy}})_0$ the subspace of the space
of functions from $\overline{\cy}$ to $\Q_p$ where the sum of values
is $0$, and $B=(\Q_p^{\overline{\cz}})^0$  the quotient of the
corresponding space for $\cz$ by the
subspace of constant functions.

(for the construction, the mixed extension has to have a global source, but we only care about the particular geometric case so we do not need to discuss the global origin).

Recall now the Fontaine functors $\Dcris$, $\Dst$ and $\DR$. These go from $p$-adic representations of $\operatorname{Gal}(\overline{K}/K)$ to certain enriched vector spaces (see~\cite{Nek93} for an excellent introduction). For a representation $E$, $\DR(E)$ is a filtered $K$-vector space while $\Dst(E)$ is a $K_0$-vector space ($K_0$ is the maximal unramified extension of $\Qp $ inside $K$) with a semi-linear Frobenius operator $f$ and a $K_0$-linear monodromy operator $N$ and $\Dcris(E)=\Dst(E)^{N=0}$. Furthermore, there is an injective $K$-linear map $\Dst(E)\otimes K\to \DR(E)$ which is bijective if $E$ is semi-stable. According to the conjectures of Fontaine, proved by Faltings~\cite{Fal89} we have

\begin{equation}
  \DR(V)=\hdr^1(X/K)\;,\; \DR(E_2)=\hdr^1(X; \cz/K)\;,\text{ and } F^0\DR(E_2)=F^1 \hdr^1(X; \cz/K) \;.\label{Faltings}
\end{equation}

 We note that the following assumption holds for $V$ and is required in the general definition of the height pairing as we will see below.
\begin{equation}\label{Vass}
  \Dcris(V)^{f=1} =\Dcris(V^\ast(1))^{f=1} = 0\;.
\end{equation}

The height pairing is now defined as follows: First we observe that some of the constructions depend on a choice of the branch of the logarithm (used in the embedding of $B_{st}$ to $B_{dR}$, see~\cite[1.3]{Nek93}). We will simply take the branch given to us by~\eqref{eq:ass7}.  The leftmost long column of~\eqref{mixext} gives a (semi-stable) extension class
$[E]\in \Hom(A,\hst^1(K,E_2))$.
\newcommand{\ksp}{(K^\times)^{(p)}}
The top short exact sequence in~\eqref{mixext} induces the following diagram with exact rows and columns
\begin{equation*}
  \xymatrix{
    & & 0 \ar[d] & 0 \ar[d] & \\
 & & \Dst(E_2)^{f=1} \ar[r]^{\sim} \ar[d]^{(0,N,-i)} & \Dst(V)^{f=1} \ar[d]^{(N,-i)} & \\
    0 \ar[r] & B\oplus B\otimes K \ar[r] \ar[d]^[@!-90]{\sim} & B\oplus \Dst(V(-1))^{f=1} \oplus DR(E_2)/F^0 \ar[r] \ar[d] & \Dst(V(-1))^{f=1} \oplus DR(V)/F^0 \ar[r]\ar[d] & 0 \\
0 \ar[r] & \hst^1(K,B(1))\ar[r] & \hst^1(K,E_2)\ar[r]\ar[d] &\hst^1(K,V)\ar[r]\ar[d]  & 0\\
   & & 0  & 0  &
    }
\end{equation*}
\newcommand{\coker}{\operatorname{coker}}
(this appears on page 157 of subsection~3.5 in~\cite{Nek93}). Note that the $B$ summand on the leftmost column maps isomorphically on itself via the horizontal map. One shows that the assumptions~\eqref{Vass} on $V$ imply that the embeddings of the $\DR$ components give isomorphisms
\begin{align*}
    DR(E_2)/F^0 &\isom \coker \left(\Dst(E_2)^{f=1}\to \Dst(V(-1))^{f=1}
    \oplus DR(E_2)/F^0 \right) \\ DR(V)/F^0 &\isom \coker \left(\Dst(V)^{f=1}\to
    \Dst(V(-1))^{f=1} \oplus DR(V)/F^0 \right)
\end{align*}
and we therefore obtain from the diagram above the following commutative diagram with exact rows
\begin{equation}\label{hstseq}
  \xymatrix{
    0 \ar[r] & B\oplus B\otimes K \ar[r] \ar[d]^[@!-90]{\sim} & B \oplus \DR(E_2)/F^0 \ar[r] \ar[d]^[@!-90]{\sim} & \DR(V)/F^0 \ar[r] \ar[d]^[@!-90]{\sim} & 0 \\
    0 \ar[r] & \hst^1(K,B(1)) \ar[r] & \hst^1(K,E_2) \ar[r] \ar@<1ex>@{.>}[l]^w & \hst^1(K,V) \ar[r] & 0
    }
\end{equation}
The leftmost vertical map is induced from an isomorphism~\cite[1.35]{Nek93}
\begin{equation}
  \label{eq:kummerst}
  \hst^1(K,\Qp(1)) \to \Qp \oplus K\;.
\end{equation}
Note the splitting $w$ to be discussed below~\eqref{wplit}.
\begin{lemma}[{\cite[1.15 and 1.19]{Nek93}}]\label{exp}
  If we remove the $B$ summands in~\eqref{hstseq} the induced vertical maps are just the Bloch-Kato
  exponential maps on the subspaces $H_f$.
\end{lemma}

The subspace $W$ provides a splitting of the short exact sequence
\begin{equation}\label{wprime}
\xymatrix{
   0 \ar[r] & B\otimes K \ar[r] & \DR(E_2)/F^0  \ar[r] \ar@<1ex>@{.>}[l]^{w^\prime}  & \DR(V)/F^0 \ar[r] & 0\;,
}
\end{equation}
hence also of the two horizontal sequence in~\eqref{hstseq} as follows:
 \begin{equation}\label{fN}
   \DR(V)/F^0 \xrightarrow{W} \DR(V) \xrightarrow{(f,N)} \DR(E_2) \to \DR(E_2)/F^0\;.
 \end{equation}
Here, the leftmost map is the splitting according to the direct sum decomposition $\DR(V)=F^0 \oplus W $ while the map denoted $(f,N)$ indicates the unique $(f,N)$-equivariant splitting. The understanding of this splitting will be the subject of the sections to follow. By~\eqref{Faltings} it is dual, via Poincar\'e duality, to the unique Frobenius and monodromy equivariant splitting
\begin{equation}
  \label{fN1}
  \xymatrix{
    {\hdr^1(X/K)}  \ar[r]^{\operatorname{res}}  &  {\hdr^1(X-\cz /K)} \;. \ar@<1ex>@{.>}[l]^{(f,N)}
}
\end{equation}
Note in particular the splitting
\begin{equation}
  \label{wplit}
  w: \hst^1(K,E_2) \to \hst^1(K,B(1))
\end{equation}
appearing in~\eqref{hstseq}.

The Kummer map determines an isomorphism $\hst^1(K,\Qp(1))\xrightarrow{\sim} \ksp$, with $\ksp$ the $p$-completion of $K^\times$, hence
\begin{equation*}
  \hst^1(K,B(1)) \xrightarrow{\sim} B\otimes \ksp\;.
\end{equation*}

\begin{definition}[{\cite[7.4]{Nek93}}]
  The local height pairing $h'(y,z)$ is minus the composition
  \begin{equation*}
    h_{A,B}^\prime : A\xrightarrow{E} \hst^1(K,E_2) \xrightarrow{w}  \hst^1(K,B(1)) \to  B\otimes \ksp \xrightarrow{1\otimes \ell_v} B
  \end{equation*}
 evaluated on the classes $y \in A$ and $z\in \Hom(B,\Qp)$. 
\end{definition}
\begin{lemma}\label{gmcase}
  The composition $\ksp \xrightarrow{\textup{Kummer}} \hst^1(K,\Qp(1)) \xrightarrow{\eqref{eq:kummerst}}\Qp \oplus K $ is the map $x\mapsto (v(x),\log(x))$ with $v$ the completed valuation and $\log$ the chosen branch of the $p$-adic logarithm.
\end{lemma}
\begin{proof}
  \cite[1.35]{Nek93}
\end{proof}
By assumption, we have the factorization~\eqref{eq:ass7} so as a map from $\Qp \oplus K $, the map $\ell_v$  factors via the projection $\pi_2$ on the second coordinate. The following is then immediate
\newcommand{\somemap}{\Theta}
\begin{corollary}
   The map $h_{A,B}^\prime \in \Hom(A,B)$ equals minus the composition 
  \begin{equation}\label{Thetaeq}
    \xymatrix{
      A \ar[r]^{[E]} \ar @/_1pc/ [rrr]_{\somemap} & \hst^1(K,E_2) \ar[r]^{\eqref{hstseq}} & B \oplus \DR(E_2)/F^0  \ar[r]^{\pi_2} & \DR(E_2)/F^0 \ar[r]^{w^\prime} & B\otimes K \ar[r]^{1\otimes \tr} & B
}
  \end{equation}
\end{corollary}

To complete the proof the the main theorem we need to compare the maps $\somemap$ and $w'$ above with the maps appearing in the Coleman-Gross construction. In the rest of this section we will do this for the map $\somemap$ using a slight modification of the arguments in~\cite{Bes99a}. In fact, we will show that Proposition~5.1 there continues to hold when we replace Coleman integration with Vologodsky integration.

By~\eqref{Faltings} $\DR(E_2)/F^0 = \hdr^1(X; \cz/K)/F^1$ and we therefore have the duality
\begin{equation}
 \Hom(\DR(E_2)/F^0,K) \isom F^1 \hdr^1(X - \cz/K)= H^0(X,\Omega_{X/K}^1(\log \cz))\;,\label{Duality}
\end{equation}

 via the cup product on de Rham cohomology, and the analogue of~\cite[Proposition~5.1]{Bes99a} is the following result.
\begin{proposition}
  Via the duality above a divisor of degree zero $D$ on $X$ supported on $\cy$ gets mapped via $\somemap$ to the functional
  \begin{equation*}
    \somemap(D): H^0(X,\Omega_{X/K}^1(\log \cz)) \to K\;,\; 
    \somemap(D)(\omega) = \int_D \omega \;,
  \end{equation*}
  with $\int_D$ the Vologodsky integral.
\end{proposition}
\newcommand{\gm}{\mathbb{G}_m}
\begin{proof}
Recall the generalized Jacobian
$J$ of $X$ with respect to $\cz$. It sits in a short exact sequence
\begin{equation}\label{shortJ}
  0\to T \to J\to J_X\to 0\;,
\end{equation}
where $J_X$ is the Jacobian of $X$.  Since we are assuming that $\cz$ splits over
$K$ we have $T\isom  (\bigoplus_{x\in \cz} \gm)/\gm$. Just as in~\cite[Prop.~5.3]{Bes99a} We have $E_2\isom T_p(J)\otimes \Qp$ and a  commutative diagram
\begin{equation}\label{diagonal}
  \xymatrix{
    A \ar[r]^-{[E]} \ar[dr] & \hst^1(K,E_2) \\
     & J(K) \ar[u]_{\textup{Kummer}}
}
\end{equation}
where the diagonal maps sends a zero divisor to its class in $J(K)$. Next, recall that the Vologodsky integral equals in this case the Zarhin integral, which is defined by
\begin{equation}
  \int_D \omega = \log_J(D)(\omega)\;,
\end{equation}
where the right hand side means the logarithm of $J$
\begin{equation}
 \log_J: J(K) \to T_0(J)\isom \hdr^1(X; \cz/K)/F^1
\end{equation}
evaluated at the class of the divisor $D$ and then paired with $\omega$ via the duality~\eqref{Duality}. In view of~\eqref{Thetaeq} and~\eqref{diagonal} The proof will be complete if we show the commutativity of
\begin{equation}\label{keydiag}
\xymatrix{
J(K)\ar[r]^{\text{Kummer}} \ar[d]^{\log_J} &  \hst^1(K,E_2) \ar[d]^{\pi_2\circ \eqref{hstseq}} \\
{ \hdr^1(X; \cz/K)/F^1} \ar[r]^{\isom} & \DR(E_2)/F^0
}
\end{equation}
To see this, we first note that by Lemma~\ref{exp} the map
\begin{equation*}
  \eqref{hstseq}|_{\DR(E_2)/F^0}:  \DR(E_2)/F^0 \to  \hst^1(K,E_2)
\end{equation*}
is just the Bloch-Kato exponential map. Therefore, by~\cite[Example~3.10.1]{Blo-Kat90} the diagram
\begin{equation*}
  \xymatrix{
    { \DR(E_2)/F^0} \ar[r]^{\isom} \ar[d]^{ \eqref{hstseq}} &  T_0(J)\ar[d]^{\operatorname{exp}}\\
     {\hst^1(K,E_2)} & J(K) \otimes \Qp  \ar[l]^{\text{Kummer}}
    }
\end{equation*}
commutes, which, given the fact that the exponential is locally an inverse of the logarithm, proves that~\eqref{keydiag} commutes at least at a neighborhood of the identity element. Then, by the uniqueness of the log map (\cite[Chapitre III, 7.6]{Bou72} and \cite{Zar96}) the diagram commutes on the subgroup $J(K)_f$ defined there. Finally, as Zarhin notes, after multiplying by an appropriate integer, each element of $J(K)$ lies in the sum of $J(K)_f$ and the image of $T(K)$. Thus, the commutativity is reduced to the same statement for $T$, which is just Lemma~\ref{gmcase}.
\end{proof}

\section{The double index}
\label{sec:double}

In this section we recall the theory of the double index~\cite{Bes98b,Bes00} and in particular the alternative definition of the projection $\Psi$ of~\eqref{Psimap} given in~\cite{Bes00}

We will fix a branch of the logarithm, e.g., the one given to us by~\eqref{diagonal}
\begin{definition}\label{annulus}
  An annulus over $K$ is a rigid analytic space $e$ isomorphic to a
  space of the form $\{r<|z|<s\}$ via a function $z$ which we call a
  uniformizer.
\end{definition}

For an annulus $e$ with a uniformizer $z$ we consider the space $\aloga(e)=\OO(e) + K\cdot \log(z)$, viewed as a subspace of the space of locally analytic functions on $e$,  and the surjective differential $d: \aloga(e) \to \Omega^1(e)$. We note that $\aloga(e)$ is independent of the choice of the uniformizer though it does depend on the choice of the branch of the logarithm. To make some statements below uniform, we also consider annuli with infinitesimal width around a point $e$, where the corresponding space is $\aloga(e) = K((z))+ K\cdot \log(z)$ and $\Omega^1(e) = K((z)) dz$, with $z$ a uniformizing parameter at the point.

According to Coleman~\cite[Lemma~2.1 and the following paragraph]{Col89}, annuli can be given two opposite \emph{orientations} corresponding to residue homomorphisms $\res: \Omega^1(e) \to K$, factoring via cohomology. A uniformizer $z$ determines an orientation by the condition $\res dz/z = 1$ and clearly the uniformizer $z^{-1}$ determines the reverse orientation. Annuli of width $0$ around a point have a canonical orientation corresponding to $\res dz/z =1$ for a uniformizer. An annulus together with a choice of orientation is called an \emph{oriented annulus}. For an oriented annulus $e$ we let $\tau(e)$ be the annulus with reverse orientation and we have the residue homomorphism
\begin{equation}
  \label{eq:orientation}
  \res_e: \Omega^1(e) \to K \text{ with } \res_{\tau(e)} = -\res_e
\end{equation}
\begin{definition}\label{doubdef}
  The double index on an oriented annulus $e$ is the unique anti-symmetric bilinear form
  $\pair{~,~}: \aloga(e) \times \aloga(e) \to K$ with the property that
  $\pair{F,G} = \res_e F d G$ whenever the left hand side has a meaning, i.e., when $\res_e d F = 0$.
\end{definition}
When needed, we denote the pairing relative to the oriented annulus $e$ by $\pair{~,~}_e$. The pairing relative to the width zero annulus around a point $x$ will be denoted $\pair{~,~}_x$.
The following is immediate
\begin{lemma}\label{locdoubprop}
  The double index has the following properties.
  \begin{enumerate}
  \item If $c$ is a constant function then $\pair{c,G} = c \res dG = -\pair{G,c} $\label{locdoubprop1}
  \item With respect to reversing orientation we have $\pair{~,~}_{\tau(e)} = - \pair{~,~}_{e}$.
  \end{enumerate}
\end{lemma}

One defines two related global indices by summing all the local indices. One uses Vologodsky integration and the other Coleman integration. To describe the first, let $X$ again denote a smooth and proper curve over $K$ and let $\cz$ be a subscheme consisting of a finite number of $K$-rational points.
\begin{definition}
  The global index pairing on $X-\cz$ is defined by
  \begin{equation}
    \label{eq:globalpair}
    \pair{\omega,\eta}_{\gl}:= \sum_{x\in \cz} \pair{F,G}_x\;,\; \omega,\eta \in H^0(X-\cz,\Omega^1)
  \end{equation}
  Where $F$ and $G$ are Vologodsky integrals of $\omega$ and $\eta$ respectively.
\end{definition}
Hidden in this definition is the fact that the pairing depends only on $\omega$ and $\eta$ and not on the choice of particular integrals $F,G$, which is a consequence of~\eqref{globpropb} of Proposition~\ref{globprop} below. Note that the definition above differs in emphasis from the one given in~\cite[(3.1)]{Bes00} where the pairing was between any two meromorphic forms, as here we have fixed the possible singularities.

The following Proposition summarizes several properties of the global 
\begin{proposition}\label{globprop}
  The global pairing satisfies the following properties.
  \begin{enumerate}
  \item It is alternating.
  \item We have $\gpair{d f, \eta}=0 $ for any rational function $f$. Consequently, the global pairing induces a well defined pairing
    \begin{equation}
      \label{eq:gpairdr}
      \gpair{~,~}: \hdr^1(X-\cz/K)\times \hdr^1(X-\cz/K) \to K
    \end{equation}\label{globpropb}
  \item The following diagram commutes
    \begin{equation}
      \label{eq:gglrest}
      \xymatrix{
        {\hdr^1(X-\cz/K)}\times \hdr^1(X-\cz/K) \ar[r]^-{\gpair{~}} & K \\
        {\hdr^1(X/K)}\times \hdr^1(X/K) \ar[u] \ar[ur]^{\cup} & 
        }
    \end{equation}
  \end{enumerate}
\end{proposition}
\begin{proof}
The first statement is obvious, the second and third are~\cite[Lemmas~3.3,3.4]{Bes00} respectively.
\end{proof}
The following is a slight reformulation of~\cite[Definition~3.9]{Bes00}
\begin{definition}\label{psidef}
  We define a map $\Psi: \hdr^1(X-\cz) \to \hdr^1(X)$ to be the adjoint of the restriction map with respect to the global and cup product pairings
  \begin{equation*}
    \Psi(x) \cup y = \gpair{x,y|_{X-\cz}}
  \end{equation*}
\end{definition}
We note that in~\cite{Bes00} the map, called $\Psi'$ there, was defined on meromorphic forms rather than cohomology classes, but the two definitions are clearly equivalent. The map $\Psi$ generalizes the logarithm for the universal vectorial extension as the following Theorem, which is~\cite[Proposition 3.11]{Bes00}, shows.
\begin{theorem}
  Let $\eta$ be a form of the third kind on $X$, holomorphic on $X-\cz$. Then $\Psi([\eta]) = \Psi(\eta)$, where $[\eta]\in \hdr^1(X-\cz/K)$ is the cohomology class of $\eta$ while $\Psi(\eta)$ is as in~\cite{Col-Gro89}.
\end{theorem}
This now completes the definition of the extension of the Coleman-Gross height pairing using Vologodsky integration.

To end this section we review the second global pairing. It is defined on Coleman's \emph{wide open spaces}, certain rigid analytic spaces whose (slightly simplified) definition can be given as follows: Suppose that the curve $X$ has an integral model $\XX$ which is smooth and proper over $\OO_K$. and pick a finite collection $S$ of $k$-rational points of $\XX\otimes k$. to each such point $x$ corresponds the \emph{residue disc} of $X$-points reducing to $x$, denoted $D_x$ and isomorphic to a standard unit disc via a uniformizer $z_x$. The wide open space $U$ is obtained by removing from $X$ the rigid subspace
\begin{equation}
  \rigS = \bigcup_{x\in S} \{|z_x|\le 1-\epsilon\}\label{discoid}
\end{equation}
for some $1-\epsilon$ close to $1$ (the space $\rigS$ is called a ``Discoid domain'' by Coleman). The ``leftover'' annuli $1>|z_x|>1-\epsilon$ are called the \emph{annuli ends} of $U$ and they are naturally oriented by their uniformizer.
We denote the finite set of annuli ends of $U$ by $\End(U)$. For future reference we call $\XX\otimes  k -S$ the \emph{reduction} of $U$.
\begin{definition}
  Let $U$ be a wide open space.
  The global index pairing for $1$-forms on $U$ is defined by
  \begin{equation}
    \label{eq:globpair}
    \pair{\dd F,\dd G}_{\gl}:= \sum_{e\in \End(U)} \pair{F,G}_e\;,
  \end{equation}
  Where $F$ and $G$ are Coleman integrals of $dF$ and $dG$ respectively.
\end{definition}
For this definition to make sense one needs to note that Coleman integration is defined for holomorphic forms on $U$ and that their restrictions to each annuli end $e$ belongs to $\aloga(e)$.
Analogous properties to the ones of the global pairing above hold here as well.
\begin{proposition}\label{grig}
  The global pairing satisfies the following properties.
  \begin{enumerate}
  \item It is alternating.
  \item We have $\gpair{d f, \eta}=0 $ for any $f\in \OO(U)$. Consequently, the global pairing induces a well defined pairing
    \begin{equation}
      \label{eq:gpairdru}
      \gpair{~,~}: \hdr^1(U)\times \hdr^1(U) \to K
    \end{equation}
  \item The pairing is compatible with restriction to a smaller $U$.
  \item If the reduction of $U$ is a rational curve, then the pairing is trivial.\label{grig4}
  \end{enumerate}
\end{proposition}
\begin{proof}
The first statement is obvious, the second is~\cite[Lemma~4.7]{Bes98b} and the third follows because the index is trivial around an annulus inside a disc where both forms extend holomorphically. Property~\eqref{grig4} is a consequence of the fact that $\gpair{~,~}$ factors via a projection on the de Rham cohomology of $X$~\cite[Prop.~4.10]{Bes98b}, and when the reduction of $U$ is rational this $X$ may be taken to be a projective line which has no first de Rham cohomology.
\end{proof}
\begin{remark}\label{extend}
  The statement above continues to hold if we let $U$ be a domain obtained by removing discs and points. In particular, if we shrink such $U$ by removing instead of a finite number of points a disc that contains all of them the global index remains the same. This follows by an argument similar to the one in the proof of~\cite[Prop.~5.5]{Bes98b} using~\eqref{grig4} of Proposition~\ref{grig}.
\end{remark}

\section{The global pairing for curves with semi-stable reduction}
\label{sec:globsemi}

In this section obtain a pairing on cohomology of curves with semi-stable reduction, which will turn out to be the same as the global pairing~\eqref{eq:globalpair}.

We begin by collecting a few well known facts about graph cohomology. For us, a graph $\Gamma$ consists of a set of vertices $V=V(\Gamma)$ and oriented edges $E=E(\Gamma)$. Each edge $e\in E(\Gamma)$ has a tail and head, $e^+,e^- \in V(\Gamma)$ respectively, and for each edge we have the edge with reverse orientation $-e$ such that $(-e)^+ = e^-$ and $(-e)^- = e^+$
\newcommand{\Xgrph}{\Gamma(X)}
\begin{definition}
  Given an abelian group $A$ we define $0$ and $1$ cochains on $\Gamma$ with values in $A$ by
  \begin{equation*}
    C^0(\Gamma,A) = \{f: V(\Gamma) \to A\}\;,\;
    C^1(\Gamma,A) = \{f: E(\Gamma) \to A\;,\; f(-e) =-f(e)\}
  \end{equation*}
  and we define the differential $d: C^0(\Gamma,A) \to  C^1(\Gamma,A)$ by
  \begin{equation*}
    d f (e) = f(e^+) -f(e^-)\;.
  \end{equation*}
  This gives the graph complex, set in degrees $0$ and $1$, whose cohomology we call the graph cohomology with values in $A$, $H^\ast(\Gamma,A)$.
\end{definition}
\begin{definition}\label{pointwise}
  The pointwise product of two $1$-cochains $c$ and $d$ on a finite graph $G$ with values in a ring $A$ is given by
  \begin{equation*}
    c\cdot d = \sum_{e\in E(G)/\pm} c(e)\cdot d(e)\;.
  \end{equation*}
  Here, the sum is over the quotient set of unoriented edges $E(G)/\pm $ which can be done because the summand is invariant to switching the orientation.
\end{definition}
\newcommand{\hh}{\mathcal{H}}
\begin{definition}
  We define the differential $d^\ast: C^1(G,A) \to C^0(G,A) $ by the formula
  \begin{equation*}
    d^\ast f(v) = \sum_{e^+ = v} f(e)\;.
  \end{equation*}
The kernel of $d^\ast$ is the space of \emph{Harmonic cochains} with values in $A$,
\begin{equation*}
  \hh(\Gamma,A) := \ker d^\ast
\end{equation*}
\end{definition}
The following is well known
\begin{theorem}\label{thmharmonic}
Suppose $A$ is a field of characteristic $0$. Then one has
\begin{equation*}
  C^1(\Gamma,A) = \hh(\Gamma,A) \oplus d C^0(\Gamma,A)
\end{equation*}
Consequently, each element in $H^1(\Gamma,A)$ has a unique \emph{harmonic representative} in $\hh(\Gamma,A)$.
\end{theorem}
This theorem allows us to define the pointwise product on cohomology. Note that the pointwise product of a harmonic cochain and a boundary is $0$.
\begin{definition}\label{pointcoh}
  The pointwise product on $H^1(\Gamma,K)$, where $K$ is a field of characteristic $0$, is defined to be the product~\ref{pointwise} of any two representing cochains, one of which at least is harmonic.
\end{definition}

Returning to geometry now suppose that $X$ is the generic fiber of a proper $\OO_K$ scheme $\XX$ with semi-stable reduction
\begin{equation}
  \label{eq:Yi}
  T = \cup_v T_v\;.
\end{equation}
Let $\Xgrph$ be the dual graph of $T$ (this is of course an abuse of notation as it really depends on the particular model). The vertices correspond to the components $T_v$ while the edges are ordered pairs of intersecting components $(T_v,T_w)$ oriented from $v$ to $w$.

\newcommand{\red}{\operatorname{red}}
The reduction map $X\to T$ allows us to split $X$ into rigid analytic domains $U_v = \red^{-1} T_v$ which are clearly wide open spaces. These then intersect along annuli corresponding bijectively to the unoriented edges of $\Xgrph$ and we can make a bijection between oriented edges and oriented annuli. We need to choose a convention for this bijection and we do this as follows:
\begin{definition}\label{convention1}
  The orientation of the annulus corresponding to and edge $e$ is the one
  corresponding to it being the annulus end of $T_{e^+}$.
\end{definition}

We next recall the description of Coleman-Iovita~\cite{ColIov99,ColIov10} for the de Rham cohomology of $X-\cz$ and for its Frobenius and monodromy operators (Here we mean a linear Frobenius. Coleman and Iovita refine this by considering the semi-linear Frobenius on a $K_0$-lattice). When describing the Frobenius and monodromy we need to make the following assumption.
\begin{assumption}\label{split-div-ass}
  the subscheme $\cz$ splits as a union of $K$-rational points, each residing in a different residue disc. In particular, it is the special fiber of an $\OO_K$-subscheme $\cZ\subset \XX$ consisting of a disjoint union of sections.
\end{assumption}
The first de Rham cohomology of $X-\cz$ is the same as that of $X$, with log poles at $\cz$, and under Assumption~\ref{split-div-ass} also of the rigid space $X-\rigZ$, where $\rigZ$ is a Coleman ``Discoid domain''~\eqref{discoid}, which follows from the comparison between de~Rham and rigid cohomology~\cite{Bal-Chi94}. This cohomology may therefore be computed as the first cohomology of the total complex of the double complex
\begin{equation*}
  \xymatrix{
    {\bigoplus_{v\in V}} \OO(U_v-\cz) \ar[r]^d \ar[d] & \bigoplus_{v\in V} \Omega^1(U_v-\cz) \ar[d] \\ \bmin_{e\in E} \OO(e) \ar[r]^d & \bmin_{e\in E} \Omega^1(e)\;.
  }
\end{equation*}
Here, $\bmin$ is the $-$ eigenspace for the involution corresponding to switching the orientation of the edges and the vertical arrows are \v{C}ech differentials. Explicitly, a one cochain is given by
\begin{equation}\label{cocycle}
  \left((\omega_v\in \Omega^1(U_v-\cz))_{v\in V},(f_e\in \OO(e))_{e\in E} \right)\; \text{such that } f_{-e} = -f(e) \text{ and } d f_e = \omega_{e^+}|_e - \omega_{e^-}|_e
\end{equation}
and a one coboundary is a cochain of the form
\newcommand{\rest}{\operatorname{rest}}
\begin{equation}
  \left((d f_v), (f_{e^+}-f_{e^-})\right)\; \text{with } f_v \in \OO(U_v-\cz)\;.
\end{equation}
One easily gets the following exact sequence
\begin{equation}\label{mayer}
  \bigoplus_{v\in V} K \to  \bmin_{e\in E} K \to \hdr^1(X-\cz/K) \xrightarrow{\rest}
   \bigoplus_{v\in V} \hdr^1(U_v-\cz) \to  \bmin_{e\in E} \hdr^1(e)
\end{equation}
From the definition of graph cohomology we can easily get out of this the short exact sequence
\begin{equation}\label{mayer1}
  0\to H^1(\Xgrph,K) \to  \hdr^1(X-\cz/K) \to \Ker
  \left(\bigoplus_{v\in V} \hdr^1(U_v-\cz) \to  \bmin_{e\in E} \hdr^1(e)\right) \to 0
\end{equation}
Suppose now that Assumption~\ref{split-div-ass} holds. Then we may replace in all of the above $\cz$ by $\rigZ$ and we can describe the Frobenius structure. The Frobenius acts as $1$ on the graph cohomology to the left. On the right we have a subobject of $\oplus_{v\in V} \hdr^1(U_v-\rigZ) $ and each summand is given a Frobenius action via its identification with the rigid cohomology of $T_v-\cZ_s$. The weights for this action are $1$ and $2$, so one can define the action on $ \hdr^1(X-\rigZ /K) $ by defining the unique Frobenius equivariant splitting of the sequence~\eqref{mayer1}, which is the projection on the weight $0$ part obtained from the weight decomposition. According to~\cite[Section 2.2]{ColIov10}  this is given by the formula, which makes sense even without Assumption~\ref{split-div-ass},
\newcommand{\phsp}{\chi}
\begin{equation}\label{phsplit}
  \phsp((\omega_v)_{v\in V},(f_e)_{e\in E})(e) = f_e-\left(F_{\omega_{e^+}}|_e- F_{\omega_{e^-}}|_e\right)
\end{equation}
The monodromy operator also takes values in $H^1(\Gamma,K) \subset \hdr^1(X-\cz/K)$ and is given by the formula
\begin{equation}
N((\omega_v)_{v\in V},(f_e)_{e\in E})(e) = \res_e(\omega_{e^+})
\end{equation}
Here one needs to note that this is the same as $\res_e(\omega_{e^-})$ because the difference between the two forms is bounded by $d f_e$ and the resulting cochain satisfies the antisymmetry condition by~\eqref{eq:orientation}. Recall the convention of Definition~\ref{convention1} for the orientation on the annulus $e$.

We will now construct a pairing on $\hdr^1(X-\cz/K)$, which will later turn out to be the same as the global index pairing of~\eqref{eq:globalpair}. We begin with an auxiliary pairing. To define it, let $\omega,\eta \in H^0(X-\cz,\Omega^1)$ and pick as auxiliary data, for every vertex $v\in V$ Coleman integrals $F_v$ and $G_v$ of $\omega $ and $\eta$ on $U_v-\cz$. We then define
\begin{equation}\label{auxil}
  \apair{\omega,\eta} = \sum_v \sum_{x\in \cz\cap U_v} \pair{F_v,G_v}_x\;.
\end{equation}
Note that this pairing depends, in general, on the auxiliary choices of Coleman integrals.
\begin{lemma}\label{cupaux}
  Suppose $\omega,\eta$ are of the second kind, with cohomology classes $[\omega],[\eta]\in \hdr^1(X/K) $. Then
  \begin{equation*}
     \apair{\omega,\eta} = [\omega]\cup [\eta]\;,
  \end{equation*}
  and is in particular independent of the auxiliary choices.
\end{lemma}
\begin{proof}
In fact, by~\eqref{locdoubprop1} of Lemma~\ref{locdoubprop} as $\res_x \eta=0 $ the local pairing $\pair{F_v,G_v}_x$ does not depend on the constant of integration appearing in $F_v$, so we might as well pick any local integral of $\omega$ at $x$, hence
\begin{equation*}
    \apair{\omega,\eta} = \sum_v \sum_{x\in \cz} \res_x \left(\eta\cdot \int \omega \right)
\end{equation*}
which is the familiar expression for the cup product.
\end{proof}
Let us now assume that $\omega$ and $\eta$ are general . Keeping the auxiliary choices as above we associate with the forms cochains on $\Xgrph$ by assigning, for an oriented edge $e$ the value 
\begin{equation}
  c(e)=F_\omega^{e^-}|_{e} - F_\omega^{e^+}|_{e} \in K\label{eq:cassigned}\;.
\end{equation}
The form $\omega$ has an associated cohomology class $[\omega] \in \hdr^1(X-\cz /K)$, which is given in the representation~\eqref{cocycle} by
\begin{equation}\label{thkind}
  \left((\omega_{U_v})_{v\in V},(0)\right)
\end{equation}
so the following is immediate from \eqref{thkind} and \eqref{phsplit}.
\begin{proposition}\label{creps}
  The cochain $c$ represents the graph cohomology class $\phsp([\omega])$.
\end{proposition}
Let $d$ be a cochain similarly associated with $\eta$.
\begin{theorem}\label{globcupthm}
We have
\begin{equation*}
   \apair{\omega,\eta}
  = \sum_v \gipair{\omega,\eta}
 +  \sum_{e\in E/\pm }\left( c(e) \res_e \eta   -d(e)\res_e \omega\right)\;.
\end{equation*}
where $\gipair{~,~}$ is the global pairing on $U_v-\cz$~\eqref{eq:globpair} (as extended by Remark~\ref{extend})
\end{theorem}
\begin{proof}
We add $ \sum_v \left(\sum_{e^+=v} \pair{F_v,G_v}_{e}\right) $ to both sides of~\eqref{auxil}. On the right hand side, we clearly get, for each vertex $v$ the global index on $U_v-\cz$ hence
\begin{equation*}
  \apair{\omega,\eta} +
  \sum_v \left(\sum_{e^+=v} \pair{F_v,G_v}_{e}\right)
  = \sum_v \gipair{\omega,\eta}\;.
\end{equation*}
To analyze the second sum on the left hand side of the equation, note that each annulus occurs twice with reverse orientation. Thus, by Lemma~\ref{locdoubprop} and Theorem~\ref{beszer} this sum becomes the following sum over the set $E/\pm$ of unoriented edges (see Definition~\ref{pointwise}). 
\begin{equation*}
  \sum_{e\in E/\pm } \pair{F_{e^+},G_{e^+}}_e - 
  \pair{F_{e^-},G_{e^-}}_e = 
  \sum_{e\in E/\pm } \left(\pair{F_{e^+},G_{e^+}}_e -
 \pair{F_{e^+}+c(e),G_{e^+}+d(e)}_e \right)
\end{equation*}
and each term in brackets becomes, by bilinearity and Lemma~\ref{locdoubprop} again
\begin{equation*}
-\pair{c(e),G_{e^+}}_e
-\pair{F_{e^+},d(e)}_e
-\pair{c(e),d(e)}_e
= d(e)\res_e \omega - c(e) \res_e \eta\;.
\end{equation*}
This completes the proof.
\end{proof}
The above Theorem suggests the following.
\begin{definition}\label{hpairdef}
     We define the pairing $\hpair{~,~}$ on $\hdr^1(X-\cz/K)$ by
    \begin{equation*}
      \hpair{\alpha,\beta} =  \sum_v \gipair{~,~}(\rest(\alpha),\rest(\beta)) + \chi(\alpha)\cdot N(\beta) - \chi(\beta)\cdot N(\alpha)\;,
    \end{equation*}
    where the first summand refers to the sum of the global pairings on the $U_v$'s applied to the image of $\alpha$ and $\beta$ under the restriction map in~\eqref{mayer} and the last two terms are the pointwise products in graph cohomology, Definition~\ref{pointcoh}, applied to the images of $N$ and $\chi$ viewed as maps into graph cohomology. 
\end{definition}
\begin{proposition}\label{hpall}
  The following holds:
  \begin{enumerate}
  \item \label{hp1} Let $\omega,\eta \in H^0(X-\cz,\Omega^1)$ with cohomology classes $[\omega],[\eta]\in \hdr^1(X-\cz/K) $ and choose the auxiliary Coleman integrals such that the corresponding cochains $c$ and $d$ as in~\eqref{eq:cassigned} are harmonic (which can be done by Theorem~\ref{thmharmonic}). Then
    $\hpair{[\omega],[\eta]}=\apair{\omega,\eta}$.
  \item \label{hp2} For any $\alpha,\beta\in \hdr^1(X/K)$ we have $\alpha\cup \beta = \hpair{\alpha|_{X-\cz},\beta|_{X-\cz}}$.
  \item \label{hp25}  The pairing $\hpair{~,~}$ does not change if we blow up a point of the special fiber.
  \item \label{hp3} The pairing $\hpair{~,~}$ is compatible with both Frobenius and monodromy in the following sense:
    \begin{enumerate}
    \item $\gpair{\phi^\ast \alpha,\phi^\ast \beta} = q \gpair{\alpha,\beta}$.
    \item $\gpair{N \alpha,\beta}+\gpair{\alpha,N\beta} =0$.
    \end{enumerate}
    Here, the Frobenius and monodromy operators are coming from the comparison with \'etale cohomology, so Assumption~\ref{split-div-ass} is not made.
  \end{enumerate}
\end{proposition}
\begin{proof}
Assertion~\eqref{hp1} follows, using Proposition~\ref{creps}, because the definition~\ref{pointcoh} of the pointwise product on graph cohomology precisely requires to take a harmonic representative to one of the terms. Then, assertion~\eqref{hp2} is immediate from Lemma~\ref{cupaux} where auxiliary choices do not matter. For~\eqref{hp25} consider first blowing up a smooth point. In this case there is an additional rational component and an additional edge $e$ to the graph, but the global pairing corresponding to this component is $0$ by~\eqref{grig4} of Proposition~\ref{grig} and because we already have a Coleman function defined there it is easy to see that we have $c(e)=d(e)=0$, so the pairing does not change. On the other hand, if we blow up an intersection point of two components we split and edge $e$ into two edges $e_1$ and $e_2$ with a vertex corresponding to a rational component. The rational component again contributes a $0$. Assuming that the harmonicity condition is imposed on the cochains $c$ and $d$ we see that harmonicity forces $c(e_1)=c(e_2)=c(e)/2$ and the same for $d$, while $\res_e=\res_{e_1}=\res_{e_2}$ so again the pairing does not change. We prove~\eqref{hp3} under Assumption~\ref{split-div-ass}, and then use~\eqref{hp25} together with the fact that we may separate sections by blowing up.
The first property follows from the corresponding property for the global pairings on wide opens~\cite[Prop.~4.10]{Bes98b}, the functoriality of the restriction maps and the equations $\chi \circ \phi^\ast = \chi$ ($\chi$ is a projection on the Frobenius invariant part) and $N\circ \phi^\ast = q \circ N$. For the second property we note that the image of $N$ is in graph cohomology, hence vanishes under restriction by~\eqref{mayer}. Also, $N^2=0$ and $\chi\circ N=N$ because the image of $N$ is already in the Frobenius invariant part, so that from corollary~\ref{cortothm} we get
\begin{equation*}
  \hpair{N \alpha,\beta} = N\alpha \cdot N \beta\;,\;
  \hpair{\alpha,N \beta} = -N\alpha \cdot N \beta\;,
\end{equation*}
completing the proof.
\end{proof}
\begin{corollary}\label{psinot}
  Define a map $\Psi_0: \hdr^1(X-\cz) \to \hdr^1(X)$ to be the adjoint of the restriction map with respect to the pairing $\hpair{~,~}$ and cup product pairings (same as in Definition~\ref{psidef} but with $\hpair{~,~}$ instead of $\gpair{~,~}$). Then $\Psi_0$ is the unique Frobenius and monodromy equivariant splitting of the restriction map  $\hdr^1(X) \to \hdr^1(X-\cz) $.
\end{corollary}
\begin{proof}
Uniqueness was proved by \nekovar. We may thus show equivariance with respect to a power of Frobenius. The equivariance now follows easily from part~\eqref{hp3} of the last Proposition.
\end{proof}

\section{End of the proof}
\label{sec:end}

We now recall the relation between Vologodsky integration and Coleman integration.

\begin{theorem}[\cite{Bes-Zer13}]\label{beszer}
  Let $\omega$ be a meromorphic form on $X$ and let $F_\omega$ be a Vologodsky integral of $\omega$. Then, for each vertex $v$ of $\Xgrph$ there exist Coleman integrals $F_\omega^v$ of $\omega|_{U_v}$ such that $F_\omega $ equals $F_\omega^v$ on $U_v(K)$. For an oriented edge $e$ let $c(e)=F_\omega^{e^-}|_{e} - F_\omega^{e^+}|_{e} \in K$. Then $c$ is a harmonic cochain on $\Xgrph$ with values in $K$,  $c\in \hh(\Xgrph,K)$.
\end{theorem}
\newcommand{\Kbar}{\overline{K}}
\begin{remark}
This somewhat cryptic formulation is required because the Vologodsky integral is defined as a function on the $K$ points only. In particular, stated this way it it is not at all defined on the annuli $e$. One can make it a function on the $\Kbar$ points, but then it is no longer the restriction of a Coleman function on the annuli. Note furthermore that the cochain $c$ depends on the choice of branch of the logarithm.
\end{remark}

\begin{corollary}\label{cortothm}
    We have $\gpair{~,~}=\hpair{~,~}$ on $\hdr^1(X-\cz/K)$.
\end{corollary}
\begin{proof}
This follows immediately from Theorem~\ref{beszer}, \eqref{hp1} of Proposition~\ref{hpall} and~\eqref{auxil}.
\end{proof}
\begin{corollary}
  The two maps $\Psi,\Psi_0: \hdr^1(X-\cz) \to \hdr^1(X)$ of Definition \ref{psidef} and Corollary\ref{psinot} respectively are equal. In particular, by Corollary~\ref{psinot}, $\Psi$ is the unique Frobenius and monodromy equivariant splitting of the restriction map  $\hdr^1(X) \to \hdr^1(X-\cz) $.
\end{corollary}

\begin{proof}[{Proof of Theorem~\ref{mainthm}}]
Having shown, under split divisors assumptions, that $\somemap $ is given by Vologodsky integration, while the Frobenius monodromy equivariant splitting $(f,N)$ from~\eqref{fN1} is given by the projection $\Psi$, the proof, including the analysis of field extensions, is now identical with the corresponding proof of Theorem~3.3 in~\cite{Bes99a}
\end{proof}

\end{document}